\documentclass[12pt]{amsart}
\usepackage[left=1in,top=1in,right=1in,bottom=1in]{geometry}

\usepackage{amsthm,amsmath,amsfonts,amssymb}
\usepackage{mathrsfs}
\usepackage{url}
\usepackage[pdftex]{graphicx}
\usepackage[colorlinks=true,linkcolor=blue,urlcolor=blue, citecolor=blue]{hyperref}
\usepackage[draft]{fixme}
\usepackage{amscd}
\usepackage{xcolor}
\usepackage{ifpdf} 

\ifpdf
\else
\fi

\numberwithin{equation}{section}
\numberwithin{table}{section}

\theoremstyle{plain}
\newtheorem{thm}{Theorem}[section]
\newtheorem{lem}[thm]{Lemma}

\newtheorem{prop}[thm]{Proposition}
\newtheorem{cor}[thm]{Corollary}
\theoremstyle{definition}

\newtheorem{definition}[thm]{Definition}

\newtheorem{exmp}[thm]{Example}
\theoremstyle{remark}
\newtheorem{rem}[thm]{Remark}

\makeatletter
\newif\if@restonecol
\makeatother

\usepackage[boxed,vlined,lined]{algorithm2e}
\usepackage{algorithmic}

\makeatletter
\@namedef{subjclassname@2010}{%
 \textup{2010} Mathematics Subject Classification}
\makeatother

\newcommand \NN {\mathbb{N}} 
\newcommand \QQ {\mathbb{Q}} 
\newcommand \ZZ {\mathbb{Z}}

\newcommand \RR {\mathbb{R}}  

\newcommand \PP {\mathbb{P}}

\newcommand{\sminus}{\ensuremath{\!\setminus\!}}
 
\newcommand{\bideg}{{\operatorname{bideg}}}

\renewcommand{\geq}{\geqslant}
\renewcommand{\leq}{\leqslant}

\DeclareMathOperator{\conv}{\operatorname{conv}}

\DeclareMathOperator{\mmult}{\operatorname{mm}}
\DeclareMathOperator{\vol}{\operatorname{vol}}
\DeclareMathOperator{\area}{\operatorname{area}}
 \DeclareMathOperator{\perim}{\operatorname{ perim}}
 
  \DeclareMathOperator{\codim}{\operatorname{codim}}
  \DeclareMathOperator{\cay}{\operatorname{Cay}}

\newcommand \ww {\omega}

\newcommand{\MV}{\operatorname{MV}}

\newcommand{\Hcal}{\mathcal{H}}

\newcommand{\Ecal}{\mathcal{E}}
\newcommand{\Pcal}{\mathcal{P}}

\newcommand{\init}{\operatorname{in}}

\newcommand{\zspan}[1]{\ensuremath{\ZZ\!\cdot\! #1}}
\newcommand{\rspan}[1]{\ensuremath{\RR\!\cdot\!#1}}

\title{Mixed Discriminants}

\author[E.\ Cattani]{Eduardo Cattani}
\author[M.\ A.\ Cueto]{Mar\'{\i}a Ang\'elica Cueto} 
\author[A.\ Dickenstein]{Alicia Dickenstein}
\author[S.\ Di Rocco]{\\ Sandra Di Rocco}
\author[B.\ Sturmfels]{Bernd Sturmfels}


\begin{document}

\begin{abstract}
The mixed discriminant of $n$ Laurent polynomials in $n$ variables
is the irreducible polynomial in the coefficients which vanishes whenever 
two of the roots coincide. The Cayley trick expresses
the mixed discriminant as an $A$-discriminant.
We show that  the degree of the mixed discriminant
is a piecewise linear function in the Pl\"ucker
  coordinates of a mixed Grassmannian. 
 An explicit  degree formula is given for the case of plane~curves.  
\bigskip

\centerline{\em Dedicated to the memory of our friend Mikael Passare
(1959--2011)}
  \end{abstract}

\maketitle
\section{Introduction}
\label{sec:introduction}

A fundamental topic in mathematics 
and its applications is the study of systems of $n$ polynomial equations in $n$ unknowns
$x = (x_1,x_2,\ldots,x_n)$ over an algebraically closed field~$K$:
\begin{equation}
\label{eq:givensystem}
f_1(x) \,=\, f_2(x) \,=\, \cdots \, = \, f_n(x) \,\,=\,\, 0 . 
\end{equation}
Here we consider Laurent polynomials with fixed support sets
$A_1,A_2,\ldots,A_n \subset \mathbb{Z}^n$:
\begin{equation}\label{eq:efs} \qquad
f_i(x) \,= \, \sum_{a \in A_i} c_{i,a} x^a \qquad \qquad (i=1,2,\ldots,n).
\end{equation}
If the coefficients $c_{i,a}$ are generic then,
according to {\em Bernstein's Theorem}~\cite{Bernstein},
 the number of solutions of~\eqref{eq:givensystem} in the
algebraic torus $(K^*)^n$ equals the {\em mixed volume}
$\,\MV(Q_1,Q_2,\ldots,Q_n)$ of the Newton polytopes $Q_i =
{\conv}(A_i)$ in $\mathbb{R}^n$.  However, for
special choices of the coefficients $c_{i,a}$, two or more of these
solutions come together in $(K^*)^n$ and create a point of higher
multiplicity.  The conditions under which this happens are encoded 
in an irreducible polynomial in the coefficients, whose zero locus 
is the {\em variety of ill-posed systems} \cite[\S I-4]{SS}.  
While finding this polynomial is
 usually beyond the reach of symbolic computation, it is often 
 possible to describe some of its invariants.  Our aim here is to characterize its degree.

An isolated solution $u \in (K^*)^n$ of~\eqref{eq:givensystem} is a {\em
  non-degenerate multiple root} if the $n$ gradient
  vectors $\nabla_x f_i(u)$ are linearly dependent, but any $n-1$ of them are linearly independent.  
This condition means that $u$ is a regular point on the curve defined
by any $n-1$ of the equations in~\eqref{eq:givensystem}.  We define the 
\emph{discriminantal variety} as the closure of the locus of coefficients $c_{i,a}$ for which the associated system~\eqref{eq:givensystem} has a non-degenerate multiple root.  If the discriminantal variety is a hypersurface, we define
the {\em mixed discriminant} of the system~\eqref{eq:givensystem} to
be the unique (up to sign) irreducible polynomial $\Delta_{A_1,\ldots,A_n}$
with integer coefficients  in the unknowns $c_{i,a}$ which defines it.
Otherwise we say that the system is \emph{defective} and set 
$\Delta_{A_1,\ldots,A_n} = 1$.

In the non-defective case, we may 
identify $\Delta_{A_1,\ldots,A_n}$ with an
{\em $A$-discriminant} in the sense of Gel'fand, Kapranov and
Zelevinsky~\cite{GKZ}.
$A$ is the {\em Cayley matrix}~\eqref{eq:cayleymatrix} of $A_1,\ldots,A_n$.
This matrix has as columns the vectors in the lifted configurations
$e_i \times A_i  \in \ZZ^{2n}$  for $i=1,\ldots,n$.
The  relationship between $\Delta_{A_1,\ldots,A_n}$
and the $A$-discriminant will be made precise in
Section~\ref{sec:2}.

In Section~\ref{sec:3} we focus on the case $n=2$. 
 Here, the mixed discriminant
$\Delta_{A_1,A_2}$ 
 expresses the condition for two
plane curves $\{f_1 = 0\}$ and $\{f_2 = 0\}$ to be tangent at a common
smooth point. 
In Theorem~\ref{thm:formula1} we present a general formula  
for the bidegree $(\delta_1,\delta_2)$ of $\Delta_{A_1,A_2}$.
In nice special cases, to be described in Corollary~\ref{cor:bideg-bound}, that 
formula simplifies to
\begin{equation}
\label{eq:deltadelta}
\begin{aligned}
 \delta_1 &\,=\,  \area (Q_1+Q_2) - \area (Q_1) - \perim (Q_2), \\
  \delta_2 &\,=\,  \area (Q_1+Q_2) - \area (Q_2) - \perim (Q_1) ,
\end{aligned}
\end{equation}
where $Q_i$ is the convex hull of $A_i$ and $Q_1 + Q_2$ is their
Minkowski sum.  The {\em area} is normalized so that a primitive
triangle has area $1$. The {\em perimeter} of $Q_i$ is the cardinality
of $ \partial Q_i \cap \ZZ^2$.

The formula~\eqref{eq:deltadelta} applies in the classical case, 
where $f_1$ and $f_2$ are dense polynomials of
degree $d_1$ and $d_2$. Here, $\Delta_{A_1,A_2}$ is the classical
 \emph{tact invariant}~\cite[$\S$96]{Salmon} whose bidegree equals
  \begin{equation}
\label{eq:tactdegree}
(\delta_1,\delta_2) \,\,\, = \,\,\,\bigl(d_2^2+2d_1d_2-3d_2, d_1^2+2d_1d_2-3d_1\bigr).
\end{equation}
See Benoist~\cite{Benoist}  and Nie~\cite{JiawangDiscr} for
the analogous formula for $n$ dense polynomials in $n$ variables.
The right-hand side of~\eqref{eq:deltadelta} is always an upper bound for the bidegree of the mixed discriminant,
but in general the inequality is strict. Indeed,  consider two sparse polynomials
\begin{equation}
\label{eq:deltadelta22}
 f_1(x_1,x_2) \,  =\,  c_{10} + c_{11} x_1^{d_1} + c_{12} x_2^{d_1} \quad \hbox{and} \quad
f_2(x_1,x_2) \,  =\,  c_{20} + c_{21} x_1^{d_2} + c_{22} x_2^{d_2},
\end{equation}
  with $d_1$ and $d_2$ positive coprime integers, then the bidegree drops from~\eqref{eq:tactdegree} to
\begin{equation}
\label{eq:deltadelta2}  
  \bigl(d_2^2 + 2d_1d_2 - 3d_2\cdot\min\{d_1,d_2\}, \,d_1^2 + 2d_1d_2 - 3d_1\cdot\min\{d_1,d_2\} \bigr).
\end{equation}

In Section~\ref{sec:4} we prove that the degree of the mixed discriminant,
in the natural $\ZZ^n$-grading, is a 
 piecewise polynomial function in the coordinates of the points
in $A_1,A_2,\ldots,A_n$.

\begin{thm}\label{thm:degreeMixed}
  The degree of the mixed discriminant cycle is 
a piecewise linear function in
  the Pl\"ucker coordinates on the mixed Grassmannian
  $G(2n,\mathcal{I})$.  It is linear on the tropical matroid strata of
  $G(2n,\mathcal{I})$ determined by the configurations $A_1, \ldots,
  A_n$. The formula on each maximal stratum is unique modulo the
  linear forms on $\wedge^{2n} \RR^m$ that vanish on
  $G(2n,\mathcal{I})$.
  \end{thm}

  Here, the {\em cycle} refers to the mixed discriminant raised to a
  power that expresses the index in $\ZZ^n$ of the sublattice affinely
  spanned by $A_1 \cup \cdots \cup A_n$.  The mixed Grassmannian
  $G(2n,\mathcal{I})$ parameterizes all $2n$-dimensional subspaces of
  $\mathbb{R}^m$ that arise as row spaces of Cayley
  matrices~\eqref{eq:cayleymatrix} with $m = \sum_{i=1}^n |A_i| $
  columns, and $\mathcal{I}$ is the partition of $\{1, \dots, m \}$
  specified by the $n$ configurations. This Grassmannian is regarded
  as a subvariety in the exterior power $\wedge^{2n} \RR^m$, via the Pl\"ucker embedding
  by the maximal minors of the matrix~\eqref{eq:cayleymatrix}.  See
  Definition~\ref{def:mG} for details.  The mixed Grassmannian admits
  a combinatorial stratification into \emph{tropical matroid strata},
  and our assertion says that the degree of the mixed discriminant
  cycle is a polynomial on these
  strata. The proof of Theorem~\ref{thm:degreeMixed} is based on {\em
    tropical algebraic geometry}, and specifically on the
  combinatorial construction of the tropical discriminant~in~\cite{DFS}.

\section{Cayley Configurations} \label{sec:2}

Let $A_1,\ldots,A_n$ be configurations in $\ZZ^n$, defining
Laurent polynomials as in~\eqref{eq:efs}. We shall relate the mixed discriminant 
$\Delta_{A_1,\ldots,A_n}$ to the $A$-discriminant, where $A$ is the Cayley matrix 
\begin{equation}
\label{eq:cayleymatrix}
 A \quad = \quad \operatorname{ Cay}(A_1,\ldots,A_n) \,\, = \,\,
\begin{pmatrix}
{\bf 1} & {\bf 0} & \cdots & {\bf 0} \\
{\bf 0} & {\bf 1} & \cdots & {\bf 0} \\
\vdots & \vdots & \ddots & \vdots \\
{\bf 0} & {\bf 0} & \cdots & {\bf 1} \\
A_1 & A_2 & \cdots & A_n 
\end{pmatrix}.
\end{equation}
This matrix has $2n$ rows and $m = \sum_{i=1}^n |A_i |$ columns,
so ${\bf 0} = (0,\ldots,0)$ and ${\bf 1} = (1,\ldots,1)$
denote row vectors of appropriate lengths.
We introduce $n$ new variables $y_1,y_2,\ldots,y_n$ and encode
our system~\eqref{eq:givensystem} by one auxiliary polynomial
with support in $A$, via the \emph{Cayley trick}: 
$$ \phi(x,y) \quad = \quad y_1 f_1(x) + y_2 f_2(x) + \,\cdots\, + y_n f_n(x). $$
We denote by $\Delta_A$ the {\em $A$-discriminant} as defined in
\cite{GKZ}.  That is, $\Delta_A$
is the unique (up to sign) irreducible polynomial
with integer coefficients in the unknowns $c_{i,a}$ which vanishes
whenever the hypersurface $\{(x,y) \in (K^*)^{2n} \,: \, \phi(x,y) = 0
\} $ is not smooth. Equivalently, $\Delta_A$ is the defining equation of the 
dual variety $(X_A)^\vee$ when this variety is a hypersurface.  
Here, $X_A$ denotes the projective toric variety
in  $ \PP^{m-1}$  associated with the Cayley matrix $A$.  If  $(X_A)^\vee$ is not a hypersurface,
then no such unique polynomial exists. We then set
$\Delta_A = 1$ and refer to $A$ as a {\em defective} configuration.  
It is useful to keep track of the lattice index
\begin{equation*}
i(A) \,\,= \,\, i(A,\ZZ^{2n}) \,\, = \,\, \left[\ZZ^{2n}\,:\,\zspan{A}\right]\!,
\end{equation*}
where $\zspan{A}$ is the $\ZZ$-linear span of the columns of $A$.
The \emph{discriminant cycle} is the polynomial
\begin{equation*}
\tilde\Delta_A \ =\ \Delta_A^{i(A)}.
\end{equation*}
The same construction makes sense for the mixed discriminant 
  and it results in the \emph{mixed discriminant cycle}
  $\tilde \Delta_{A_1,\ldots,A_n}$.
  The exponents $i(A)$ will be compatible in the following theorem.

\begin{thm}\label{thm:Adiscr}
The mixed discriminant equals the
$A$-discriminant of the Cayley matrix:
$$\Delta_{A_1,\ldots,A_n} \,\,= \,\, \Delta_A.$$
\end{thm}

This result is more subtle than it may seem at first glance. It implies that 
$(A_1,\ldots,A_n)$ is defective if and only if $A$ is defective.
The two discriminantal varieties can differ in that~case.

\begin{exmp} \label{ex:caveat}
Let $n = 2$ and consider the Cayley matrix
$$ A \,\,\, = \,\,\, \begin{pmatrix} {\bf 1} & {\bf 0} \\
  {\bf 0} & {\bf 1} \\
  A_1 & A_2 \end{pmatrix} \,\,\, = \,\,\,
  \begin{pmatrix}
   1 & 1 & 1 & 0 & 0 & 0 \\
    0 & 0 & 0 & 1 & 1 & 1 \\
   0 & 1 & 2 & 0 & 0 & 0 \\
   0 & 0 & 0 & 0 & 1 & 2 
   \end{pmatrix}
$$
The corresponding system~\eqref{eq:givensystem} consists of two
univariate quadrics in different variables:
$$ f_1(x_1) \, =\, c_{10} + c_{11} x_1 + c_{12} x_1^2 \,=\,0 \quad  \hbox{and} \quad
     f_2(x_2) \, =\, c_{20} + c_{21} x_2 + c_{22} x_2^2 \, = \,0. $$
This system cannot have a non-degenerate  multiple root,
for any choice of coefficients $c_{ij}$, so the
$(A_1,A_2)$-discriminantal variety  is empty.
On the other hand, the $A$-discriminantal variety is non-empty.
It has codimension two and is defined by
$   c_{11}^2  - 4 c_{10} c_{12} = 
        c_{21}^2  - 4 c_{20} c_{22} =  0 $. \hfill$\Diamond$
\end{exmp}

\begin{proof}[Proof of Theorem~\ref{thm:Adiscr}]
We may assume $i(A)=1$.
Let $u \in (K^*)^n$ be a non-degenerate multiple root
of  $f_1(x) = \cdots = f_n(x)= 0$. Our definition ensures the
existence of a  unique (up to scaling) vector $v \in (K^*)^n$
such that $\sum_{i=1}^n v_i \nabla_x f_i(u) $ is the zero vector.
The pair $(u,v) \in (K^*)^{2n}$ is a singular point of
the hypersurface defined by $\phi(x,y) = 0$. By projecting into the space of 
coefficients $c_{i,a}$, we see that
the $(A_1,\ldots,A_n)$-discriminantal variety is  contained in the
$A$-discriminantal variety. Example~\ref{ex:caveat} shows that this
containment can be strict.

We now claim that $\Delta_A \not=1$ implies $\Delta_{A_1,\ldots,A_n} \not=1 $.
This will establish the proposition because
 $\Delta_{A_1,\ldots,A_n}$ is a factor of $\Delta_A$,
 and $\Delta_A$ is  irreducible, so the two discriminants are equal.
 Each point $(u,v) \in (K^*)^{2n}$ defines a point on $X_A$.
If $\Delta_A \not= 1$, the dual variety $(X_A)^\vee$
  is a hypersurface in the dual projective space $(\PP^{m-1})^\vee$. Moreover, see e.g.\ \cite{Katz}, 
 a generic hyperplane in the dual variety is tangent to the
 toric variety $X_A$ at a single point.
  
Consider a generic point on the conormal variety of $X_A$ in
 $\PP^{m-1} \times (\PP^{m-1})^\vee$.  It is represented by a pair
 $\bigl( (u,v) , \,c \bigr)$, where $(u,v) \in (K^*)^{2n}$ and $c$ is
 the coefficient vector of a polynomial $\phi(x,y)$ such that $(u,v)$
 is the unique singular point on $\{\phi(x,y) = 0\}$.  The coefficient
 vector $c$ defines a point on the $(A_1,\ldots,A_n)$-discriminantal
 variety unless we can relabel such that the gradients of $f_1,\ldots ,f_{n-1}$ are
 linearly dependent at $u$. Assuming that this holds, we let
\[
\sum_{i=1}^{n-1} t_i \nabla_x f_i(u) \,\,= \,\, 0 
\]
be the dependency relation and set $t=(t_1,\ldots,t_{n-1},0)\neq
\mathbf{0}$. The point $\bigl( (t+u,v) , \,c \bigr)$ lies on the
conormal variety of $X_A$. This implies that the generic hyperplane
defined by $c$ is tangent to $X_A$ at two distinct points $(u,v)\neq
(t+u,v)$, which cannot happen. It follows that
$\Delta_{A_1,\ldots,A_n} \not=1$, as we wanted to show. This concludes
our proof.
\end{proof}

\begin{exmp} \label{ex:hyperdet}
Let $n=2$ and $A_1 = A_2 =
\{(0,0),(1,0), (0,1),(1,1)\}$, a unit square. Then
\begin{equation*}
   \begin{aligned}
     f_1(x_1,x_2) &  =  a_{00} + a_{10} x_1 + a_{01} x_2 + a_{11} x_1 x_2, \\
     f_2(x_1,x_2) &  =   b_{00} + b_{10} x_1 + b_{01} x_2 + b_{11} x_1 x_2.     
   \end{aligned}
  \end{equation*}
  The Cayley configuration $A$ is the standard $3$-dimensional cube. The
  $A$-discriminant  is known to be the
  {\em hyperdeterminant} of format $2 {\times} 2
  {\times} 2$, by~\cite[Chapter 14]{GKZ}, which equals
  \begin{equation*}
    \begin{array}{ccl}
  \Delta_{A_1,A_2}  & = & 
  a_{00}^2 b_{11}^2-2 a_{00} a_{01} b_{10} b_{11}-2a_{00} a_{10} b_{01} b_{11}
  -2 a_{00} a_{11} b_{00} b_{11} \\ & & + 4 a_{00} a_{11} b_{01} b_{10} 
  + a_{01}^2 b_{10}^2+4 a_{01} a_{10} b_{00} b_{11} - 
  2 a_{01} a_{10} b_{01} b_{10} \\ &  &
  - 2 a_{01}  a_{11} b_{00} b_{10} 
  + a_{10}^2 b_{01}^2- 2 a_{10} a_{11} b_{00} b_{01} + a_{11}^2 b_{00}^2. 
  \end{array}
 \end{equation*}
 Theorem~\ref{thm:Adiscr} tells us that this hyperdeterminant coincides with the
 mixed discriminant of $f_1$ and $f_2$.  Note that the bidegree equals
 $(\delta_1,\delta_2) = (2,2)$, and therefore~\eqref{eq:deltadelta} holds.  \hfill$\Diamond$
 \end{exmp}

We now shift gears and focus on defective configurations.  
We know from Theorem~\ref{thm:Adiscr} that $(A_1,\ldots,A_n)$ is defective if and only
if the associated Cayley configuration $A$ is defective.  While there has been some recent progress
on characterizing defectiveness  \cite{DFS, Esterov, matsui}, the problem of 
classifying defective configurations $A$ remains open, except in cases when the {\em codimension} of $A$ is at most four \cite{curran,DS} or when the toric variety 
$X_A$  is smooth or $\QQ$-factorial \cite{CDR, dirocco}.  
 Recall that $X_A$ is {\em smooth} if and only if, at each 
every vertex of the polytope $Q = {\rm conv}(A)$, the first elements of $A$
that lie on the incident edge directions form a basis for the lattice spanned by $A$. The variety $X_A$ is
$\QQ$-factorial when $Q$ is a \emph{simple} polytope, that is, when
every vertex of $Q$ lies in exactly ${\rm dim}(Q) $ facets.
Note that smooth implies $\QQ$-factorial.

We set $\dim(A) = \dim(Q)$, and we say that
$A$ is \emph{dense} if  $A = Q \cap \ZZ^d$. 
A subset $F \subset A$ is called a \emph{face} of $A$, denoted $F\prec A$, if $F$ is the
intersection of $A$ with a face of the polytope $Q$. We will denote by
$s_n$ the standard $n$-simplex and by $\sigma_n $ the configuration of its {\em vertices}.

When $A$ is the Cayley configuration of $A_1,\dots,A_n \subset \ZZ^n$,
the codimension of $A$ is $m-2n$. This number is
 usually rather large.  For instance, if all $n$
polytopes $Q_i = \conv(A_i) $ are full-dimensional in $ \RR^n$ then
$\codim(A) \geq n\cdot (n-1)$, and thus, for $n\geq 3$, we are
outside the range 
where defective configurations have been classified.  However, if $n=2$
and the configurations $A_1$ and $ A_2$ are full-dimensional we
can classify all defective configurations.

 \begin{prop}\label{prop:defective_two}
   Let $A_1, A_2 \subset \ZZ^2$ be full-dimensional configurations.
   Then, $(A_1,A_2)$ is defective if and only if, up to affine isomorphism, $A_1$ and  $A_2$ 
      are both translates of  $\,p\cdot \sigma_2$, for some positive
   integer $p$.
 \end{prop}
 
\begin{proof}
  Let $A=\cay(A_1,A_2)$. Both $A_1$ and $A_2$ appear as faces of $A$.
  In order to prove that $A$ is non-defective, it suffices to exhibit
  a $3$-dimensional non-defective subconfiguration (see
  \cite[Proposition~3.1]{CDSCompositio2001} or
  \cite[Proposition~3.13]{Esterov}).  Let $u_1,u_2,u_3$ be
  non-collinear points in $A_1$ and $v_1,v_2$ distinct points in
  $A_2$.  The subconfiguration $\{u_1,u_2,u_3,v_1,v_2\}$ of $A$ is
  $3$-dimensional and non-defective if and only if no four of the
  points lie in a hyperplane or, equivalently, if the vector $v_2 -
  v_1$ is not parallel to any of the vectors $u_j - u_i$, $j\not= i$.
  We can always find such subconfigurations unless $A_1$ and $A_2$ are
  the vertices of triangles with parallel edges. In the latter case,
  we can apply an affine isomorphism to get $A_1 = p\cdot \sigma_2$
  and $A_2$ a translate of $\pm q\cdot \sigma_2$, where $p$ and $q$ are positive
  integers.  The total degree of the mixed discriminant~equals
\begin{align*}
     \deg(\Delta_{p\cdot\sigma_2,q\cdot\sigma_2}) & \,\,=  \,\,
     (p^2 + q^2 + pq - 3 \min\{p,q\}^2)/{\rm gcd}(p,q)^2, \\
     \deg(\Delta_{p\cdot\sigma_2,-q\cdot\sigma_2}) & \,\,=\,\,   (p+q)^2/{\rm      gcd}(p,q)^2.
\end{align*}
The first formula follows from~\eqref{eq:deltadelta2} and it is
positive unless $p= q$.  The second formula will be derived in Example
\ref{sparse-negative}. It always gives a positive number. This concludes our
proof.
\end{proof}

Similar arguments can be used to study the case when one of the 
configurations is one-dimensional.  However, it is more instructive to
classify such defective configurations from the bidegree of the mixed discriminant.
This will be done in Section~\ref{sec:3}.

 \begin{cor}
   Let $A_1$ and $A_2$ be full-dimensional configurations in $\ZZ^2$.
   Then the mixed discriminantal variety of $(A_1,A_2)$ is either a
   hypersurface or empty.
 \end{cor}
 
\begin{rem} \label{rem:simple}
The same result holds in $n$ dimensions 
when the toric variety $X_A$  is smooth and
 $A_1,\ldots,A_n$ are full-dimensional configurations in $\ZZ^n$.
 Under these hypotheses,       $(A_1,\ldots, A_n)$ is defective if and only if
   each   $A_i$ is affinely equivalent to $p \cdot \sigma_n$, with $p \in \NN$. 
   In particular, the mixed discriminantal variety of $(A_1,\ldots,A_n)$ is either a
   hypersurface or empty.
    The  ``if'' direction is straightforward:
    we may assume $i(A)=1$ and $p=1$ by  
 replacing $\ZZ^n$ with the lattice spanned
 by $p e_1, \dots, p e_n$. Then, the system~\eqref{eq:givensystem}
 consists of linear equations, and it is clearly defective. 
 The ``only-if'' direction is derived from  results in \cite{dirocco}: $(A_1,\ldots, A_n)$ is
  defective if and only if the $(2n-1)$-dimensional 
  polytope $Q ={\rm conv}(A)$ is isomorphic to a Cayley
  polytope of at least $t +1 \ge n+1$ 
  configurations of dimension $ k < t$ that have the same normal fan.
    As  we already have a Cayley
  structure of $n$ configurations in dimension $n$, we deduce 
  $t=n$ and $k=n - 1$. Then,  we should have
  $Q\simeq  s_{n-1}\times s_{n} \simeq s_n \times s_{n-1}$. After an affine transformation,
all $n$ polytopes $Q_i$ are standard $n$-simplices and all $A_i$ are
 translates of $s_n$. This shows that $A$ has an
"inverted" Cayley structure of $n + 1$ copies of~$\sigma_{n-1}$.
 \end{rem}
 
 We expect Proposition~\ref{prop:defective_two} to hold
 in $n$ dimensions without the smoothness hypothesis in Remark~\ref{rem:simple}.
   Clearly, whenever the mixed volume of $Q_1,\ldots,Q_n$ is $1$, then there
 are no multiple roots and we have $\Delta_{A_1,\ldots,A_n} =
 1$.  The following result gives a necessary and sufficient condition
 for being in this situation: up to affine equivalence, this is just
 the linear case.
 
\begin{prop}
  If $A_1,\ldots,A_n$ are $n$-dimensional configurations in $\ZZ^n$
  then the mixed volume $ \MV(Q_1,\ldots,Q_n)$ is $ 1$ if and only if,
  up to affine isomorphism, $A_1=\cdots=A_n=\sigma_n$.
\end{prop}

\begin{proof} We shall prove the ``only-if'' direction by induction on
  $n$. Suppose $ \MV(Q_1,Q_2,\ldots,Q_n)= 1$.  By the
  Aleksandrov-Fenchel inequality, we have $\vol(Q_i)=1$ for all $i$,
  where the volume form is normalized so that the standard $n$-simplex
  has volume 1. Since the mixed volume function is monotone in each
  coordinate, for any choice of edges $l_i$ in $Q_i$ we have
\[
0\,\leq \,\MV(l_1, l_2, \ldots, l_n)\,\leq\, 
\MV(l_1, Q_2, \ldots, Q_n)\,\leq \,\MV(Q_1, \ldots, Q_n) \,= \, 1.
\]
Since all  polytopes $Q_i$ are full-dimensional, we can pick $n$
linearly independent edges $l_1, \ldots, l_n$. Therefore
$\MV(l_1,  \ldots, l_n) > 0$ and
$\MV(l_1, l_2, \ldots, l_n)=\MV(l_1, Q_2, \ldots, Q_n) = 1$.  
In particular, the edge $l_1$ has length one.
After a change of coordinates we may
assume that $l_1 = e_n$, the $n$-th standard basis vector. Denote by $\pi$
 the projection of $\ZZ^n$ onto $\ZZ^n/\zspan{e_n}\simeq
 \ZZ^{n-1}$ and the corresponding map of $\RR$-vector spaces.  
 We then have $\,\MV(\pi(Q_2), \ldots, \pi(Q_n)) = 1$.
 
 By the induction hypothesis, we can transform the first $n-1$
 coordinates so that $\pi(A_2) = \cdots = \pi(A_n) = \sigma_{n-1}$.
 This means that $A_i\subset \sigma_{n-1}\times \zspan{e_n}$.  Now,
 let $a_i$ be a point in $A_i$ not lying in the coordinate hyperplane
 $x_n=0$.  Then $1\leq \vol(\conv(\sigma_{n-1},a_i)) \leq \vol(Q_i)
 =1$, and we conclude that $\,Q_i = \conv(\sigma_{n-1},a_i)$.  But,
 since $\vol(Q_i)=1$, it follows that $a_i= b_i \pm e_n$, for some
 $b_i\in \sigma_{n-1}$.  By repeating this process with an edge of
 $A_1$ containing the point $a_1$, we see that all $b_i$'s are equal
 and that the sign of $e_n$ in all $a_i$'s is the same. This shows
 that, after an affine isomorphism, we have $A_1=\cdots=
 A_n=\sigma_n$, yielding the result.
\end{proof}

\section{Two Curves in the Plane}\label{sec:3}

In this section we study the condition for two plane curves to be tangent.
This condition is the mixed discriminant in the case $n=2$.  Our
primary goal is to prove Theorem~\ref{thm:formula1}, which gives a
formula for the bidegree of the mixed discriminant cycle of two full-dimensional
planar configurations $A_1$ and $ A_2$. Remark
\ref{one-dimensional} addresses the degenerate case
when one of the $A_i$ is one-dimensional. Our main
tool is the connection between discriminants and principal
determinants.  In order to make this connection precise, and to define
all the terms appearing in~\eqref{eq:dega1}, we recall some basic
notation and facts.  We refer to~\cite{Esterov, GKZ} for further details.

Let $A\subset \ZZ^{d}$ and $Q$ the convex hull of $A$. 
As is customary in toric geometry, we assume that  $A$ lies
in a rational hyperplane that does not pass through the
origin.  This holds for Cayley
configurations~\eqref{eq:cayleymatrix}.  Given any subset $B\subset A$
we denote by $\zspan{B}$, respectively $\rspan{B}$, the linear span of
$B$ over $\ZZ$, respectively over $\RR$.    For any face $F\prec A$
we define the \emph{lattice~index}
\begin{equation*}
i(F,A) \,\,\, := \,\,\, \left[  \rspan{ F} \cap \ZZ^d : \zspan{F} \right].
\end{equation*}
  We set $i(A) = i(A,A) = [\ZZ^d: \zspan{A}]$. We consider
  the $A$-discriminant $\Delta_A$ and the 
  {\em principal $A$-determinant} $E_A$.  
  They are  defined in \cite{GKZ} under the assumption that $i(A)=1$.
  If $i(A)>1$ then   we change the ambient lattice from $\ZZ^d$
  to $\zspan{A}$, and we define the associated \emph{cycles}
 \[
\tilde E_A = E_A^{i(A)} \quad \hbox{and} \quad \tilde \Delta_A =
 \Delta_A^{i(A)}.
\]
The expressions on the right-hand sides are computed relative to the lattice $\zspan{A}$.
 
 \begin{rem}
 The principal $A$-determinant of \cite[Chapter~10]{GKZ}  
 is a polynomial $E_A$ in the variables $c_{\alpha}$, $\alpha\in A$. Its Newton polytope is the 
 \emph{secondary polytope} of $A$, and its degree is $(d+1)\vol(\conv(A))$,
 where $\vol = \vol_{\zspan{A}}$ is the normalized lattice volume for ${\zspan{A}}$.  
 We always have $\deg(\tilde E_A) = (d+1)\vol_{\ZZ^d}(\conv(A))$,
 where $\vol_{\ZZ^d}$ is the normalized lattice volume for $\ZZ^d$.
 \end{rem}

 We state the factorization formula of Gel'fand, Kapranov and
 Zelevinsky \cite[Theorem 1.2,~Chapter 10]{GKZ} for the principal
 $A$-determinant as in Esterov \cite[Proposition~3.10]{Esterov}:
 \begin{equation}\label{eq:FactorizationE_A}
\tilde E_A\,\,=\,\,\pm\tilde\Delta_A\cdot \prod_{F\prec A} \tilde\Delta_{F
}^{u(F,A)}.
\end{equation}
The product runs over all proper faces of $A$. The exponents
${u(F,A)}$ are computed as follows.  Let $\pi$ denote the projection to
$\rspan{A}/\rspan{F}$ and $\Omega$ the normalized volume form on
$\rspan{A}/\rspan{F}$. This form is normalized  with respect to
the lattice $\pi(\ZZ^d)$,   so that the fundamental
domain with respect to integer translations has volume
$(\dim(\rspan{A})- \dim(\rspan{F}))!$.
We set  $$\,u(F,A) \,:= \,\Omega\bigl(\conv(\pi(A))\, \sminus \conv(\pi(A\sminus F)) \bigr).$$

\begin{rem}
The positive integers $u(F,A)$ are denoted $c^{F,A}$ in \cite{Esterov}.  If $i(A)=1$ then
${u(F,A)}$  is the {\em subdiagram volume} associated with $F$, as in \cite[Theorem~3.8, Chapter~5]{GKZ}.
 \end{rem}
 
 We now specialize to the case of Cayley configurations $A =
 \cay(A_1,A_2)$, where $A_1,A_2\subset \ZZ^2$ are full-dimensional.
 Here, $A$ is a $3$-dimensional configuration  in the
 hyperplane $x_1 + x_2 =1$ in $\RR^4$.  Note that $i(A,\ZZ^4) =
 i(A_1\cup A_2,\ZZ^2)$.  The configurations $A_1$ and $A_2$ are
 facets of $A$.  
 
 We say that $F$ is a \emph{vertical} face of $A$ if $F\prec A$ but $F
 \not \prec A_i$, $i=1,2$.  The vertical facets of $A$ are either
 triangles  or two-dimensional Cayley configurations defined by edges $e\prec
 A_1$ and $f\prec A_2$.  This happens if $e$ and $f$ are parallel and
 have the same orientation, that is, if they have the same inward
 normal direction when viewed as edges in $Q_1$ and $Q_2$.  We call such edges
 \emph{strongly parallel} and denote the vertical facet they define by
 $V(e,f)$.

 Let $\Ecal_i$ denote the set of edges of $A_i$ and set
\begin{equation*}
\Pcal \,\,= \,\, \{(e,f) \in \Ecal_1 \times \Ecal_2 : e \hbox{\ is strongly parallel to\ } f\}.
\end{equation*}
We write $\ell(e)$ for the {\em normalized length} of an edge $e$ with
respect to the lattice $\ZZ^2$.  For $v\in A_1$ we define
\begin{equation}
\label{eq:mmv} \mmult(v) \,\,= \,\, \MV(Q_1,Q_2) -
\MV(\conv(A_1\sminus v),Q_2),
\end{equation}
 and similarly for $v \in
A_2$. This quantity is the \emph{mixed multiplicity} of $v$ in
$(A_1,A_2)$.
   
   \begin{thm}\label{thm:formula1}
Let $A_1$ and $A_2$ be full-dimensional configurations in $\ZZ^2$.  Then
\begin{eqnarray}\label{eq:dega1}
\delta_1 := \deg_{A_1}(\tilde\Delta_{A_1,A_2})  &=& \area(Q_2) + 2 \MV(Q_1,Q_2) \nonumber \\&& - \sum_{(e,f)\in \Pcal} \min\{u(e,A_1), u(f,A_2)\}\,  \ell(f) \,\, - \! \sum_{v\in \operatorname{Vert}A_1} \!\!\mmult(v).
\end{eqnarray}
\begin{eqnarray*}
\delta_2 := \deg_{A_2}(\tilde\Delta_{A_1,A_2})  &=& \area(Q_1) + 2 \MV(Q_1,Q_2) \nonumber\\&& - \sum_{(e,f)\in \Pcal} \min\{u(e,A_1), u(f,A_2)\}\,  \ell(e) \,\,- \! \sum_{v\in \operatorname{Vert}A_2} \!\!\mmult(v).
\end{eqnarray*}
\end{thm}

Theorem~\ref{thm:formula1} is the main result in this section.  We
shall derive it from the following formula (which is immediate
from~\eqref{eq:FactorizationE_A}) for the bidegree of our mixed
discriminant:
\begin{equation}\label{eq:degMixedDiscriminant}
\bideg(\tilde{\Delta}_{A_1,A_2}) \,\,= \,\, \bideg(\tilde{E}_A)  \,- \, \sum_{k=1}^2\!\sum_{F\prec A_k}
\!\! u(F,A)\,\bideg(\tilde\Delta_{F}
) \,-  \, \!\!\! \sum_{\substack{F \prec A \\ 
\text{vertical
}}}
\!\!\!
u(F,A)\,\bideg (\tilde\Delta_{F}).
\end{equation}
Note the need for the {\em cycles} $\tilde{\Delta}_{A_1,A_2},
\tilde{E}_A$ and $\tilde \Delta_{F}$ in this formula.  We shall prove
Theorem~\ref{thm:formula1} by studying each term on the right-hand
side of~\eqref{eq:degMixedDiscriminant}, one dimension at a time.  A
series of lemmas facilitates the exposition.

\begin{lem}
  The bidegree $\bideg(E_A)$ of the principal $A$-determinantal cycle
  $\tilde{E}_A$ equals
\begin{equation}
  \label{eq:bideg_princ} (3\area (Q_1) +
  \area (Q_2) + 2 \MV(Q_1,Q_2) , \area (Q_1) + 3\area (Q_2) + 2 \MV(Q_1,Q_2)).
\end{equation}
\end{lem}
\begin{proof}
 By \cite{GKZ}, the total degree of $\tilde{E}_A$ is $4 \vol(Q)$.
From any triangulation of $A$ we can see 
$$
\vol(Q) \,\,=\,\, \area (Q_1) + \area (Q_2) +\MV(Q_1,Q_2).
$$
Examining the tetrahedra in a triangulation
reveals that the bidegree is given by~\eqref{eq:bideg_princ}.
\end{proof}

Any pyramid is a defective configuration; hence, the vertical facets of $Q$
that are triangles do not contribute to the right-hand side
of~\eqref{eq:degMixedDiscriminant} and can be safely ignored from now
on. In particular, we see that the only non-defective vertical facets
are the trapezoids $V(e,f)$ for $(e,f)\in \Pcal$. The following lemma
explains their contribution to~\eqref{eq:degMixedDiscriminant}.

\begin{lem}\label{prop:verticalfacets}
Let $V(e,f)$ be the vertical facet  of $A$ associated with $(e,f)\in \Pcal$.  Then
\begin{enumerate}
\item $\bideg (\tilde\Delta_{V(e,f)}) \,\,= \,\, (\ell(f) , \ell(e))$,
  \smallskip
\item $u(V(e,f),A) \,\,=\,\, \min\{u(e,A_1),u(f,A_2)\}.$
\end{enumerate}
\end{lem}

\begin{proof}
  The configuration $V(e,f)$ is the Cayley lift of two one-dimensional
  configurations.  Its discriminantal cycle is the resultant of two univariate
  polynomials of degree $\ell(e)$ and $\ell(f)$, so (1) holds. In
  order to prove (2), we note that $u(V(e,f), Q)$ equals the
  normalized length of a segment in $\RR^3/\rspan{V(e,f)}$ starting at
  the origin and ending at the projection of a point in $A_1$ or
  $A_2$. This image is the closest point to the origin in the line
  generated by the projection of $Q$. Thus, the multiplicity
  $u(V(e,f), A)$ is the minimum of $u(e,A_1)$ and $u(f,A_2)$.
\end{proof}

\smallskip

We next study the horizontal facets of $A$ given by $A_1$ and $A_2$.   

\begin{lem}\label{prop:horizontafacets}  
The discriminant cycle of the plane curve defined by $A_i$ has total degree
\[
\deg(\tilde\Delta_{A_i}) \,\,= \,\, 
3 \area(Q_i) - \sum_{e\in \operatorname{Edges}A_i} u(e,A_i)
\deg(\tilde\Delta_{e}) \, - \!\!\!  \sum_{v \in
  \operatorname{Vert}A_i} u(v,A_i), 
\]
where
$\, u(v,A_i) = \area(Q_i) - \area(\conv(A_i \sminus v)) $.
\end{lem}
\begin{proof}
  This is a special case of~\eqref{eq:degMixedDiscriminant} because
  $\deg(\tilde E_{A_i}) = 3\cdot \area(Q_i)$ and
  $\deg(\tilde\Delta_v)=1$ for any vertex $v\in A_i$.  The statement
  about $u(v,A_i)$ is just its definition.
\end{proof}

Next, we consider the edges of $A$.  The vertical edges are defective
since they consist of just two points.  Thus we need only examine the
edges of $A_1$ and $A_2$.

\begin{lem}\label{prop:edges} Let $e $ be an edge of $A_i$.  Then $u(e,A) = u(e,A_i)$.
\end{lem}

\begin{proof}
 Consider  the projection  $\pi :Q_i \to \RR^2/\rspan{e}$. The image
 $\pi(Q_i)$ is a segment of length
  $M_1 = \max\{\ell([0,\pi(m)]) : m\in A_i\}$, while $\conv(A_i\sminus e)$
  projects to a segment of length $M_2 = \max\{\ell([0,\pi(m)]) : m\in
  (A_i \sminus e)\}.$ Thus $u(e,A_i) = M_1 - M_2$.  
Next,  consider  the projection  $Q \to \RR^3/\rspan{e}$.
  The images of $A$ and  $\conv(A \sminus e)$ under this projection are 
  trapezoids in  $\RR^3/\rspan{e}$.
        Their set-theoretic difference is a
  triangle of height $1$ and base $M_2-M_1$.
\end{proof}

\begin{lem}\label{mult_vertex} Let $v$ be a vertex of $A_i$. Then $\,u(v,A) = u(v,A_i) + \mmult(v)$.
\end{lem}

\begin{proof}  Suppose $v\in A_1$. The volume form
$\Omega$ is normalized with respect to the lattice
$\ZZ^3$.
The volume of our Cayley polytope $\cay(A_1,A_2)$ equals $\area(Q_1) + \area(Q_2) +
\MV(Q_1,Q_2)$, and the analogous formula holds for
${\rm conv}(A\sminus v)=\cay(A_1\sminus v, A_2)$. We conclude
\begin{align*}
  u(v,A) &=  \vol(\cay(A_1,A_2)) - \vol(\cay(A_1\sminus v,A_2)) \\
  &= \area(Q_1) - \area(\conv(A_1\sminus v) + \MV(Q_1,Q_2) - \MV(\conv(A_1\sminus v),Q_2))\\
  &= u(v,A_1) + \mmult(v).\\
\end{align*}
\vskip -1.3cm
\end{proof}

\begin{proof}[Proof of Theorem~\ref{thm:formula1}]
  By symmetry, it suffices to prove~\eqref{eq:dega1}.  We start with
  the $A_1$-degree of the principal $A$-determinantal cycle
  $\tilde{E}_A$ given in~\eqref{eq:bideg_princ}.  In light
  of~\eqref{eq:FactorizationE_A}, we subtract the $A_1$-degrees of the
  various discriminant cycles corresponding to all faces of
  $A$. Besides the contribution from $A_1$, having $u(A_1,A)=1$ and
  given by Lemma~\ref{prop:horizontafacets}, only the vertices and the
  vertical facets contribute. Using 
  Lemmas~\ref{prop:edges} and~\ref{mult_vertex}, we derive the  desired
  formula.
\end{proof}

At this point, the reader may find it an instructive exercise to
derive~\eqref{eq:tactdegree} and~\eqref{eq:deltadelta2} from
Theorem~\ref{thm:formula1}, and ditto for $(\delta_1,\delta_2) =
(2,2)$ in Example~\ref{ex:hyperdet}.  Here are two further examples.

\begin{exmp}\label{dense-negative}
Let $A_1 $ and $A_2$ be the dense triangles $(d_1 s_2) \cap \ZZ^2$ and
 $(-d_2 s_2) \cap \ZZ^2$.  Here, $i(A)=1$ and $\tilde\Delta_{A_1,A_2}=\Delta_{A_1,A_2}$.  
 We have $\MV(d_1 s_2,-d_2 s_2)=2d_1d_2$ and $\Pcal = \emptyset$.  Computation of the mixed areas
in~\eqref{eq:mmv} yields 
$\mmult(v) = d_2$ for vertices $v \in A_1$
and 
$\mmult(v) = d_1$ for vertices $v \in A_2$.
 We conclude
$$\bideg(\Delta_{A_1,A_2}) \,\,=\,\, 
\bigl(d_2^2 + 4d_1d_2 - 3d_2 , \, d_1^2 + 4d_1d_2 - 3d_1 \bigr).$$
\vspace{-6ex}

\hfill$\Diamond$
\end{exmp}

\begin{exmp}\label{sparse-negative}
Let $A_1 = d_1 \sigma_2 $  and 
$A_2 = - d_2 \sigma_2$. This is the
sparse version of Example~\ref{dense-negative}.
Now, $i(A) = g^2$, where $g = \gcd(d_1,d_2)$, and
$\tilde\Delta_{A_1,A_2}=\Delta_{A_1,A_2}^{g^2}$. 
We still have  $\MV(d_1 s_2,-d_2 s_2)=2d_1d_2$ and $\Pcal = \emptyset$, 
but $\mmult(v) = d_1d_2$ for all $v \in A$.~Hence
$$\bideg(\Delta_{A_1,A_2}) \,\,=\,\, 
\frac 1{g^2}\bigl(d_2^2 + d_1d_2,\,d_1^2 + d_1d_2\bigr) .
$$
\vspace{-6ex}

\hfill$\Diamond$
\end{exmp}

\begin{rem}\label{one-dimensional}
From \eqref{eq:degMixedDiscriminant} we may also derive 
formulas for the bidegree of the mixed discriminant in the 
case when one of the configurations, say $A_2$, is one-dimensional.
The main differences with the proof of Theorem~\ref{thm:formula1} 
is that now we must treat $A_2$ as an edge, rather than a facet, 
and that it is enough for an edge $e$ of $Q_1$ to be parallel 
to $Q_2$ in order to have a non-defective vertical facet of $Q$.  
Clearly, there are at most two possible edges of $Q_1$ parallel to $Q_2$.  
The $A_2$-degree of the mixed discriminant cycle now has a very simple expression:
\begin{equation}\label{eq:one-dimensional}
\delta_2 \,\,=\,\, \area(Q_1) - \sum_{e||Q_2} u(e,A_1) \ell(e),
\end{equation}
where the sum runs over all edges $e$ of $Q_1$ which are parallel to $Q_2$.  

In particular, if no edge of $Q_1$ is parallel to $Q_2$, then $\delta_2>0$ 
and $(A_1,A_2)$ is not defective.  If  only one edge $e$ of $Q_1$ is parallel to $Q_2$ then 
$\delta_2=0$ if and only if  $\area(Q_1) = u(e,A_1) \ell(e)$ 
but this happens only if there is a single point of $A_1$ 
not lying in the edge $e$.  This means that $A=\cay(A_1,A_2)$ 
is a pyramid and hence is defective.  Finally, if 
there are two edges $e_1$ and $e_2$ of $Q_1$ parallel to $Q_2$ then $\delta_2=0$ if and only if
$\,\area(Q_1) = u(e_1,A_1) \ell(e_1) + u(e_2,A_1) \ell(e_2)$. This can
only happen if all the points of $A_1$ lie either in $e_1$ or $e_2$.  In this case,
$A$ is the Cayley lift of three one-dimensional configurations, and it is defective as well.
\end{rem}

Our next goal is to provide a sharp geometric bound for the sum of the mixed
multiplicities. We start by providing a method to compute such
invariants by means of mixed subdivisions.

\begin{lem}\label{lm:MixedMult}
Let $A_1,A_2$ be full-dimensional in $\ZZ^2$ and $v \in A_1$.
Any mixed subdivision~of $Q^*=\conv
(A_1\sminus v) +Q_2$ extends to a mixed subdivision of
$Q=Q_1{+}Q_2$.  The mixed multiplicity $\mmult(v)$ is the sum of the
Euclidean areas of the mixed cells in the 
closure $D$~of~$Q \sminus Q^*$.
\end{lem}

\begin{figure}[ht]
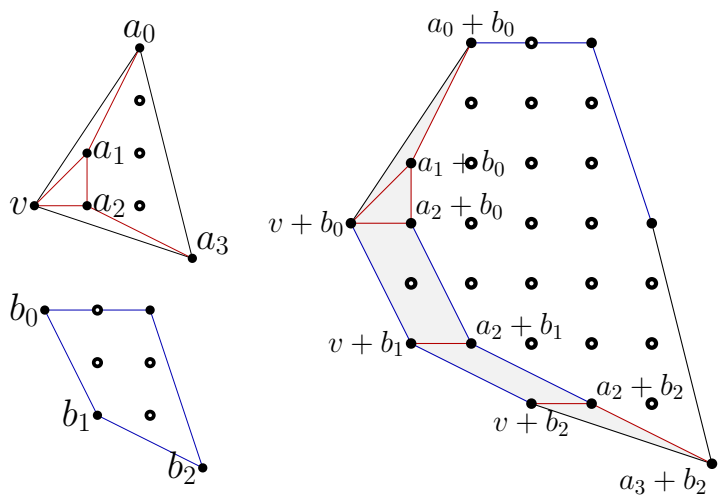

  \centering
\hfill
  \begin{minipage}{.2\linewidth}
  \includegraphics[scale=0.7]{mixedMultiplicitySubdivision.4}
\vspace{2ex}
\hspace{1ex}  \includegraphics[scale=0.7]{mixedMultiplicitySubdivision.5}
\end{minipage}
  \begin{minipage}{.5\linewidth}
    \includegraphics[scale=0.8]{mixedMultiplicitySubdivision.6}
  \end{minipage}
\caption{Geometric computation of the mixed multiplicity $\mmult(v)$
via a  suitable mixed subdivision of the two polygons.
 The region $D$ is shown in~grey.}
\label{fig:mixedMult}
\hfill\end{figure}

\begin{proof}
  Let $\Ecal'_2(v)$ denote the collection of edges in $A_2$ whose
  inner normal directions lie in the \emph{relative interior} of the
  dual cone to a vertex $v$ of $A_1$. Equivalently, $\Ecal'_2(v)$
  consists of those edges $[b, b']$ of $A_2$ such that $v+b$ and
  $v+b'$ are both vertices of $A_1+A_2$. See
  Figure~\ref{fig:mixedMult}.

First, assume $\Ecal'_2(v)=\emptyset$. Then, there exists a unique
$b\in A_2$ such that $v+b$ is a vertex of $Q$.  It follows that there
exist $a_0,\dots,a_r \in A_1$ such that $D$ is a union of triangles of the form
  \begin{equation}
\label{eq:nochange} D \,\,\,= \,\,\,\,\bigcup_{i=1}^r\,\conv(\{v+b,
  a_{i-1} + b, a_{i}+b\}),
\end{equation}
and $\{a_{i-1} + b, a_{i}+b\}$ are edges in the subdivision of $Q^*$.
Then, we can extend the subdivision of $Q^*$ by adding the triangles
in~\eqref{eq:nochange}.  This does not change the mixed areas and
$\mmult(v) =0$.

Suppose now that $\Ecal'_2(v)=\{f_1,\dots,f_s\}$, $s\geq 1$, with
indices in counterclockwise order. Let $b_0,\ldots,b_s$ be the
vertices of $A_2$ such that $f_i$ is the segment $[
b_{i-1},b_i]$.  The pairs $v+b_{i-1},v+b_i$, for $i=1,\ldots,s$,
define edges of $Q$ which lie in the boundary of $D$.  Let $a_0,a_r\in
A_1$ be the vertices of the edges of $Q_1$ adjacent to $v$.  We insert
$r-1$ points in $A_1$ to form a counterclockwise oriented sequence
$a_0,a_1,\dots,a_r $ of vertices of $\conv(A_1 \sminus v)$.  Then $a_0
+ b_0$ and $a_r+b_s$ are vertices of $Q^*$, and the boundary of $D$
consists of the $s$ segments $[v+b_{i-1},v+b_i]$, together with
segments of the form $[a_{i-1} + b_j, a_{i}+b_j]$ or $[a_i +
b_{j-1}, a_{i}+b_j]$. Figure~\ref{fig:mixedMult} depicts the case
$r=3, s=2$.  Given this data, we subdivide $D$ into mixed and unmixed
cells. The unmixed cells are triangles $\{v+b_j, a_{i+1}+b_j,
a_{i}+b_j\}$ coming from the edges $[a_i+b_j, a_{i+1}+b_j]$ of
$Q^*$. The mixed cells are parallelograms $\{v+b_j, v+b_{j+1},
a_i+b_{j+1}, a_i+b_j\}$ built from the edges $[a_i+b_j,
a_i+b_{j+1}]$ of $Q^*$. This subdivision is compatible with that of
$Q^*$.
\end{proof} 

We write $\Ecal'_1$ for the set of all edges of $A_1$ that are not
strongly parallel to an edge of $A_2$. The set
$\,\Ecal'_2$ is defined analogously.

\begin{prop}\label{thm:mult_bound} Let $A_1, A_2 \in \ZZ^2$ 
be two-dimensional configurations.  Then
  \begin{enumerate}
\item The sum of  the lengths of all edges in the set~$\Ecal'_j$ is a lower bound for 
the sum of the mixed multiplicities over all vertices of the other configuration $A_i$. In symbols,
\begin{equation*}
\sum_{v\in \operatorname{Vert}A_i} \mmult(v)\,\, \geq \,\, \sum_{e\in \Ecal'_j}
\ell(e)\, \quad {\rm for} \,\, j\not=i.
\end{equation*}
\item If $\,\Ecal'_j = \emptyset\,$ then ${\displaystyle{\sum_{v\in 
\operatorname{Vert}A_i} \mmult(v) = 0}}$.
\item If $i(A_1)=i(A_2)=1$ and the three toric surfaces corresponding to $A_1$, $A_2$ 
and $A_1+A_2$ are smooth
 then the bound in (1) is sharp.
\end{enumerate}
\end{prop}

\begin{proof}
We keep the notation of the proof of 
 Lemma~\ref{lm:MixedMult}. Recall  that the set
$\Ecal'_2$ is the union of the sets $\Ecal_2'(v)$, where $v$ runs over
all vertices in $A_1$.
  By Lemma~\ref{lm:MixedMult}, the mixed multiplicity
  $\mmult(v)$ is the sum of the Euclidean areas of the mixed cells in
  $D$.  Each mixed cell is a parallelogram $\{v+b_{k-1},v+b_k,a_i + b_{k-1},
    a_{i}+b_k\}$, so its area is $   \ell([b_{k-1},b_{k}]) \cdot \ell([v,a_i])\cdot
  |{\rm det}(\tau_{k-1},\eta_i)|$, where $\tau_{k-1}$ and $\eta_i$ are
  primitive normal vectors to the edges $[b_{k}, b_{k+1}]$ and
  $[v,a_i]$. Thus $\mmult(v) \geq \sum_{e\in \Ecal'_2(v)}\ell(e)$.  
  Since   $\Ecal'_2$ is the disjoint union of the sets $\Ecal'_2(v)$,
  summing over all vertices $v$ of $A_1$ gives the desired lower bound.
  Part (2) also follows from Lemma~\ref{lm:MixedMult},
  as $\Ecal'_2=\emptyset$ implies that
  the subdivision of $D$ has  no  mixed~cells.

  It remains to prove (3). The assumption that $X_{A_1}$ is smooth
  implies that 
 the segment
  $[a_0,a_1]$ is an edge in $\conv(A_1 \sminus v)$.  Therefore, all
  mixed cells in $D$ are parallelograms with vertices
  $\{v+b_{k-1},v+b_k,a_i + b_{k-1}, a_{i}+b_k\}$ for $i=0,1$, $k=1,
  \ldots, s$.  This parallelogram has  Euclidean
  area $|{\rm det}(a_i -
  v, b_k - b_{k-1})|$, but, since $X_{A_1+ A_2}$ is smooth, we have:
$$|{\rm det}(a_i - v, b_k - b_{k-1})| = \ell([v,a_i])\cdot
\ell([b_{k-1}, b_k]) = 1 \cdot \ell([b_{k-1}, b_k])$$ Since
$\Ecal_2'(v)=\{[b_0, b_1], \ldots, [b_{s-1}, b_s]\}$, 
this equality and Lemma~\ref{lm:MixedMult} yield the result.
\end{proof}

\begin{rem}\label{rmk:smooth}
The equality  $\sum_{v\in \operatorname{Vert}A_i} \mmult(v) =\sum_{e\in \Ecal'_j}
\ell(e)$ in case  $i(A_1)=i(A_2)=1$ and the toric surfaces of $A_1$, $A_2$ and $A_1+A_2$ are smooth,
can be interpreted and proved with tools form toric geometry. Indeed, in this case,
let $X_{1}, X_{2}$ and $X$ be the associated toric varieties. Then, there are
birational maps $\pi_i\colon X\to X_i$ defined
by the common refinement of the associated normal fans, $i=1,2$.
The map $\pi_1$ is given by successive toric blow-ups of fixed
  points of $X_1$ corresponding to vertices $v$ of $A_1$ 
  for which $\Ecal'_2(v) \not= \emptyset$. The lenghts of the corresponding
  edges occur as the intersection product of the invariant (exceptional) 
  divisor associated to the edge with the ample line bundle associated to $A_2$, 
  pulled back to $X$.
\end{rem}

If $A_i$ is dense, then it is
immediate to check that $u(e,A_i) =1$ for all edges $e\prec A_i$. We
conclude with a geometric upper bound for the bidegree of the mixed
discriminant.

\begin{cor}\label{cor:bideg-bound}
Let $A_1$ and $A_2$ be full-dimensional configurations in $\ZZ^2$.  Then:
\begin{enumerate}
\item The bidegree satisfies
$\,\deg_{A_i}(\tilde\Delta_{A_1,A_2})  \leq 
\area(Q_j) + 2 \MV(Q_1,Q_2) - \perim(Q_j)\,, \ j\not=i$.
\item Equality holds in {\rm (1)} if  $i(A_1)=i(A_2)=1$ and  the three toric surfaces of $A_1$, $A_2$ and $A_1 + A_2$ are smooth.
\item Equality holds in {\rm (1)} if $Q_1,Q_2$ have the same normal fan and one of $A_1$ or $A_2$~is~dense.
\end{enumerate}
\end{cor}

\begin{proof}
Assume $i=1$. Statement (1) follows from~\eqref{eq:dega1} and 
$$\sum_{(e,f)\in \Pcal} \min\{u(e,A_1), u(f,A_2)\}\,  \ell(f) + \sum_{v\in A_1} \mmult(v) 
\,\geq \,\sum_{f\in \Ecal_2 \sminus \Ecal'_2} \ell(f) + \sum_{f\in \Ecal'_2} \ell(f)
\,\, =\,\, \perim(A_2).$$
Statement (2) follows from  Theorem~\ref{thm:mult_bound} (3) and the fact that the smoothness condition implies $u(e,A_1) = u(f,A_2) = 1$ for all edges $e\prec A_1, \,f \prec A_2$.
Finally, if $Q_1$ and $Q_2$ have the same normal fan then
$\Ecal'_1 = \Ecal'_2 = \emptyset $, and, by
Theorem~\ref{thm:mult_bound} (2), all mixed multiplicities vanish.  
Density of $A_1$ or $A_2$ implies 
$ \min\{u(e,A_1), u(f,A_2)\} = 1$ for every pair
 $(e,f)\in \Pcal$.~Hence
$$\sum_{(e,f)\in \Pcal} \min\{u(e,A_1), u(f,A_2)\}\,  \ell(f) \,\, = \,\, {\rm perim}(Q_2).$$
\vspace{-6ex}

\end{proof}

Corollary~\ref{cor:bideg-bound} establishes
the degree formula (\ref{eq:deltadelta}).
We end this section with an example for which
 that formula 
  holds, even though conditions (2) and (3) do not.
It also shows that,
  unlike for resultants \cite[\S 6]{DFS}, the degree of the mixed discriminant
  can decrease when removing a single point from $A$ without altering the
  lattice or the convex hulls of the configurations.

\begin{exmp} \label{ex:nini}
  Consider the  dense configurations $A_1:=\{(0,0), (1,0), (1,1), (0,1)\}$
  and $A_2:=\{(0,0), (1,3), (-1,2),(0,1), (0,2)\}$.  
  The vertex $v=(0,0)$ of $A_2$ is a singular point.
     However, its mixed multiplicity equals
 $1$, so it agrees with the lattice length of the associated edge
  $[(0,0),(1,0)]$ in $A_1$.   Theorem~\ref{thm:formula1} implies that the
  bidegree of the mixed discriminant $\Delta_{A_1,A_2}$ equals 
  $(\delta_1, \delta_2) =(12,8)$.   If we remove the point $(0,1)$ from $A_2$, 
  the mixed multiplicity of $v$ is raised to $2$ 
 and  the bidegree of the mixed discriminant  decreases to $(12,7)$. 
 \hfill$\Diamond$
\end{exmp}
\medskip

\section{The Degree of the Mixed Discriminant is Piecewise Linear}
\label{sec:4}

Theorem~\ref{thm:formula1} implies that the bidegree of the mixed
discriminant of $A_1, A_2\subset \ZZ^2$ is piecewise linear in the
maximal minors of the Cayley matrix $A=\cay(A_1,A_2)$.  In this
section we prove Theorem~\ref{thm:degreeMixed} which extends the same
statement to arbitrary Cayley configurations, and we describe suitable regions
of linearity.  Theorem~\ref{thm:degreeMixed} allows us to obtain formulas for the
multidegree of the mixed discriminant by linear algebraic methods,
provided we are able to compute it in sufficiently many examples.
This may be done by using the {\em ray shooting
algorithm} of~\cite[Theorem 2.2]{DFS}, which has become a standard technique in tropical
geometry. We start by an example in dimension 3, which was computed using
Rinc\'on's software~\cite{Rincon}.

\begin{exmp}
    Consider the following three trinomials in three variables:
    \[ \begin{aligned}
        f&\,=\,  a_1\, x + a_2\, y^p + a_3 \, z^p,\\
        g&\,=\, b_1\,x^q + b_2\, y + b_3\, z^q,\\
h & \,=\, c_1 \,x^r + c_2\, y^r + c_3\, z.
      \end{aligned}
    \]
    Here $p,q$ and $r$ are arbitrary integers different from $1$.
    By the degree we mean the triple of integers that records the
    degrees of the mixed discriminant  cycle  $\Delta(f,g,h)$ in the unknowns
    $ (a_1,a_2,a_3)$, $(b_1,b_2,b_3)$, and $(c_1,c_2,c_3)$.
    It equals the following triple of piecewise polynomials:
\begin{equation*}
\begin{aligned}
\bigl(2pqr + q^2 r + q r^2 - q - r - 1  -p \min\{q,r\}, \\
   \, 2pqr + p^2r  + pr^2   - p - r - 1 - q \min\{r, p\}, \\
     \quad  2pqr + p^2q  + pq^2 - p - q - 1  - r \min\{p, q\} \bigr) \\
\end{aligned}
\end{equation*}
These three polynomials are linear functions in the
$6 \times 6$-minors of the $6 \times 9$ Cayley matrix $A$ that represents
$(f,g,h)$.
The space of all systems of three trinomials will be defined as a certain
 {\em mixed Grassmannian}.  The $6 \times 6$-minors represent its
  Pl\"ucker coordinates. \hfill$\Diamond$
\end{exmp}

Given $m \in \NN$ with $n \leq m$, consider a partition $ {\mathcal I}
= \{I_1, \ldots, I_n\}$ of the set $[m]=\{1, \ldots, m\}$.  Let
$G(d,m)$ denote the {\em affine cone over the Grassmannian} of
$d$-dimensional linear subspaces of $\RR^m$, given by its Pl\"ucker
embedding in $\wedge^d \RR^m$.  Thus $G(d,m)$ is the subvariety of
$\wedge^d \RR^m$ cut out by the quadratic Pl\"ucker relations. For
instance, for $d=2,m=4$, this is the hypersurface $G(2,4)$ in
$\wedge^2 \RR^4 \simeq \RR^6$ defined by the unique Pl\"ucker relation
$x_{12} x_{34} - x_{13} x_{24} + x_{14} x_{23} = 0$.

\begin{definition}\label{def:mG}
  The \emph{mixed Grassmannian} $G(d,\mathcal{I})$ associated to the
  partition $\mathcal I$ is defined as the linear subvariety of
  $G(d,m)$ consisting of all subspaces 
  that contain the vectors
  $e_{I_j}:=\sum_{i\in I_j} e_i$ for $j=1, \ldots, n$.  Here ``linear
  subvariety'' means that $G(d,\mathcal{I})$ is the intersection of
  $G(d,m)$ with a linear subspace of 
  the $\binom{m}{d}$-dimensional real vector space $\wedge^d \RR^m$.
\end{definition}

The condition that a subspace $\xi$ contains $e_{I_j}$ translates into
a system of $n(m-d)$ linearly independent linear forms in the Pl\"ucker
coordinates that vanish on $G(d,\mathcal{I})$.  These linear forms are
obtained as the coordinates of the exterior products $\,\xi \wedge
e_{I_j}\,$ for $j = 1,\ldots,n$.

We should stress one crucial point. As an abstract variety, the mixed
Grassmannian $G(d,\mathcal{I})$ is isomorphic to the ordinary
Grassmannian $G(d{-}n,m{-}n)$, where the isomorphism maps
 $\xi$ to its image modulo $\operatorname{
  span}(e_{I_1},\ldots, e_{I_n})$.  However, we always work with
the Pl\"ucker coordinates of the ambient Grassmannian $G(d,m)$ in
$\wedge^d \RR^m$. We do \textbf{not} consider the mixed Grassmannian 
$G(d{-}n,m{-}n)$ in its Pl\"ucker embedding in $\wedge^{d-n} \RR^{m-n}$.

Our mixed Grassmannian has a natural  decomposition
into finitely many strata whose definition  involves
\emph{oriented matroids}. On each stratum, 
the degree of the mixed discriminant cycle is a linear
function in the Pl\"ucker coordinates. 
In order to define tropical matroid strata and to prove Theorem
\ref{thm:degreeMixed}, it will be convenient to regard the mixed
discriminant as the $A$-discriminant $\Delta_A$ of the Cayley matrix
$A$. In fact, we shall consider $\Delta_A$ for arbitrary matrices $A
\in \mathbb{Z}^{d \times m}$ of rank $d$ such that $e_{[m]} =
(1,1,\ldots,1)$ is in the row span of $A$.  Then, $A$ represents a
point $\xi$ in the Grassmannian $G(d,\{[m]\}) $. This is the proper
subvariety of $G(d,m)$ consisting of all points whose subspace
contains $e_{[m]}$.

In what follows we assume some  familiarity
with matroid theory and tropical geometry.  We refer to 
\cite{DFS, FS2005} for details.
Given a $d \times m$-matrix $A$ of rank $d$ as above, we let $M^*(A)$ denote the
corresponding dual matroid on $[m]$. This matroid has rank $m-d$.
 A subset $I =\{i_1,\dots,i_r\} \subseteq [m]$ is {\em independent} 
 in $M^*(A)$ if and only if $e^*_{i_1},\dots,e^*_{i_r}$ are linearly 
 independent when restricted to $\ker(A)$, where $e_1^*,\dots,e_n^*$ 
 denotes the standard dual basis.  The  {\em flats} of the
    matroid $M^*(A)$ are the subsets
$J \subseteq [m]$ such that $[m] \sminus J$
is the support of a vector in $\ker(A)$. 

Let $\mathcal{T}({\ker}(A))$ denote the tropicalization of the kernel
of $A$. This tropical linear space is a balanced fan of dimension
$m-d$ in $\mathbb{R}^m$.  It is also known as the {\em Bergman fan} of
$M^*(A)$, and it admits various fan structures \cite{FS2005, Rincon}.
Ardila and Klivans \cite{AK} showed that the chains in the geometric
lattice of $M^*(A)$ endow the tropical linear space $
\mathcal{T}({\ker}(A))$ with the  structure of a simplicial fan.
The cones in this fan are $\operatorname{ span}(e_{J_1},e_{J_2},
\ldots, e_{J_r})$ where $\mathcal{J}=\{J_1 \subset J_2 \subset \cdots
\subset J_r\}$ runs over all chains of flats of $M^*(A)$. Such a cone
is maximal when $r = m-d-1$. Given any such maximal chain and any
index $i \in [m]$, we associate with them the following $m \times m$ matrix:
\begin{equation*}
M(A,\mathcal{J},i):=( A^T, e_{J_1}, e_{J_2}, \ldots, e_{J_{m-d-1}}, e_i ).
\end{equation*}
Its determinant is a linear expression in the Pl\"ucker coordinates 
of the row span $\xi$ of $A$:
$$\,{\det}(M(A,\mathcal{J},i))
\, \,= \,\, \xi \wedge e_{J_1} \wedge e_{J_2} \wedge \cdots \wedge
e_{J_{m-d-1}} \wedge e_i.$$

\begin{definition} \label{def:strata}
Let $A$ and $ A'$ be matrices representing  points $\xi$ and $\xi'$ in $G(d,\{[m]\})$.
These points belong to the same {\em tropical matroid stratum} if they
have the same dual matroid, i.e., 
$$ \quad M^*(A) \,= \,M^*( A'), $$ 
and, in addition,
for all $i \in [m]$ and all maximal chains of flats
$\mathcal{J}
$ in the above matroid, the determinants of the matrices
$M(A,\mathcal{J},i)$ and $M(A',\mathcal{J},i)$ have the same sign.
\end{definition}

\begin{rem}
Dickenstein {\it et al.}~\cite{DFS} gave the following formula
for the {\em tropical $A$-discriminant}:
\begin{equation}
\label{eq:tropDA}
 \mathcal{T}(\Delta_A) \quad = \quad  \mathcal{T}({\ker}(A)) \, + \, \operatorname{ rowspan}(A). 
 \end{equation}
This is a tropical cycle in $\mathbb{R}^m$, i.e.~a polyhedral fan 
that is balanced relative to the multiplicities associated to its
maximal cones. The dimension of $\mathcal{T}(\Delta_A)$ equals
 $m-1$ whenever $A$ is not defective.  It is clear from the
formula~\eqref{eq:tropDA} that $\mathcal{T}(\Delta_A)$ depends
only on the subspace $\xi = \operatorname{ rowspan}(A)$, so it is a function of $\xi \in G(d,\{[m]\})$.
The tropical matroid strata are the subsets of $G(d,\{[m]\})$ throughout which
the combinatorial type of~\eqref{eq:tropDA} does not change. 
\end{rem}

\begin{exmp}\label{exmp:strata}
We illustrate the definition of the tropical matroid strata by revisiting the formulas in
(\ref{eq:deltadelta2}) and Example~\ref{sparse-negative}.   The  Cayley matrix 
of the two sparse triangles equals
$$
 A \ =\  \operatorname{ Cay}(A_1,A_2) \,\, = \,\,
\begin{pmatrix}
1&1&1&0&0&0\\
0&0&0&1&1&1\\
0&d_1&0&0&d_2&0\\
0&0&d_1&0&0&d_2\\
\end{pmatrix}.
$$
The matroid $M^*(A)$ has rank $2$, so every maximal chain of flats in $M^*(A)$ consists of 
a single rank $1$ flat.  These flats are $J_1 = \{1,4\}$, $\,J_2 = \{2,5\}$, and $\, J_3 = \{3,6\}$.
The $6 \times 6$-determinants ${\det}(M(A,\mathcal{J},i))$  obtained by augmenting $A$ with one vector 
$e_{J_k}$ and one unit vector $e_i$ are $0$,  $\pm d_1(d_1-d_2)$, or $\pm d_2(d_1-d_2)$. 
This shows that $d_1 \geq d_2 \geq 0$ and $d_1 \geq 0 \geq d_2$ are tropical matroid strata,
corresponding to  (\ref{eq:deltadelta2}) and  to Example~\ref{sparse-negative} with
$d_2$ replaced by  $-d_2$. 
\hfill$\Diamond$ 
\end{exmp}   

\begin{rem}
  The verification that two configurations lie in the same tropical
  matroid stratum may involve a huge number of maximal flags if we use
  Definition~\ref{def:strata} as it is.  In practice, we can greatly
  reduce the number of signs of determinants to be checked, by
  utilizing a coarser fan structure on $ \mathcal{T}({\ker}(A))$.  The
  coarsest fan structure is given by the {\em irreducible flats} and
  their {\em nested sets}, as explained in \cite{FS2005}.  Rather than
  reviewing these combinatorial details for arbitrary matrices, we
  simply illustrate the resulting reduction in complexity when
  $M^*(A)$ is the {\em uniform matroid}.  This means that any $d$
  columns of $A$ form a basis of $\RR^d$.  Then, $M^*(A)$ has
  ${(m-d-1)! {{m}\choose{m-d-1}}}$ maximal flags $J_1 \subset J_2
  \subset \cdots \subset J_{m-d-1}$ constructed as follows. Let
  $I=\{i_1,\dots,i_{m-d-1}$ be an $(m-d-1)$-subset of $[m]$ and
  $\sigma$ a permutation of $[m-d-1]$.  Then, we set $J_k :=
  [m]\sminus\{i_{\sigma(1)},\dots,i_{\sigma(k)}\}$.  It is clear that
  the sign of ${\det} (\, A^T, e_{J_1}, e_{J_2}, \ldots,
  e_{J_{m-d-1}}, e_k \,)$ is completely determined by the signs of the
  determinants
$${\det} (\, A^T, e_{i_1}, e_{i_2}, \ldots, e_{i_{m-d-1}}, e_k \,),$$ 
where $i_1 < i_2 < \cdots < i_{m-d-1}$.  Hence, we only need to check
${ {{m}\choose{m-d-1}}}$ conditions. 
\end{rem}

Recall that the
\emph{$A$-discriminant cycle} $\tilde{\Delta}_A= \Delta_A^{i(A)}$
is effective of codimension $1$, provided
$A$ is non-defective. The lattice index $i(A)$ is
the gcd of all maximal minors of~$A$.

\begin{thm}\label{thm:degreeDelta}
  The degree of the $A$-discriminant cycle is piecewise
  linear in the Pl\"ucker coordinates on $G(d,\{[m]\})$. It is
  linear on  the tropical matroid strata. The formulas on
  maximal strata are unique  modulo the linear forms 
  obtained from the entries of $ \xi \wedge e_{[m]}$.
    \end{thm}

In both Theorem~\ref{thm:degreeMixed} and Theorem~\ref{thm:degreeDelta}, 
the notion of ``degree'' allows for
any grading that makes the 
respective discriminant homogeneous. For the mixed discriminant
$\Delta_{A_1,\ldots,A_n}$ we are
 interested in the $\mathbb{N}^n$-degree.
Theorem~\ref{thm:degreeMixed} will be derived as a corollary
from Theorem~\ref{thm:degreeDelta}. 

 \begin{proof}[Proof of Theorem~\ref{thm:degreeDelta}]
 The uniqueness of the degree formula follows from our earlier remark
that the entries of
$  \xi \wedge e_{[m]}$ are the linear relations on the mixed
Grassmannian $G(d,\{[m]\})$. We now show how tropical geometry leads to
the desired piecewise linear formula.

From the representation of the tropical discriminant in~\eqref{eq:tropDA},
 Dickenstein {\it et al.}~\cite[Theorem 5.2]{DFS}
derived the following formula  for the
 initial monomial of the $A$-discriminant $\Delta_A$
 with respect to any generic weight vector $\ww \in \mathbb{R}^m$.
 The exponent of the variable $x_i$ in the initial monomial ${\init}_\ww(\Delta_A)$ of
  the $A$-discriminant $\Delta_A$ is equal to
\begin{equation}\label{eq:degree1}
\sum_{\mathcal{J} \in {\mathcal C}_{i,\ww} } |\, 
      \det ( A^T, e_{J_1}, \ldots,  e_{J_{m-d-1}}, e_i )  \,| \,  .
\end{equation}
Here, $\mathcal C_{i,\ww}$ is the set of maximal chains $\mathcal{J}$
of $M^*(A)$ such that the rowspan $\xi$ of $A$ has non-zero
intersection with the relatively open cone $\RR_{> 0} \bigl\{ e_{J_1},
\ldots, e_{J_{m-d-1}}, -e_i, -\ww\bigr\}$.

It now suffices to prove the following
statement: if two matrices $A$ and $A'$ lie in the same tropical
matroid stratum, then there exists weight vectors $\ww$ and $\ww'$ such
that $\mathcal C_{i,\ww} = \mathcal C_{i,\ww'}$.  This ensures
that the sum in~\eqref{eq:degree1}, with the absolute value replaced
with the appropriate sign, yields a linear function in the Pl\"ucker
coordinates of $\xi$ for the degree of $\Delta_A$ and $\Delta_{A'}$.

The condition $\mathcal{J} \in \mathcal C_{i,\ww}$ is equivalent 
to the weight vector $\ww$ being in the cone
$$\RR_{> 0} \bigl\{ e_{J_1}, \ldots, e_{J_{m-d-1}}, -e_i\bigr\} \,\,+ \,\, \xi.$$
Hence, it is convenient to work modulo 
$\xi$.  This amounts to considering the exact sequence
\begin{equation*}
\begin{CD}
0 @>>>  \RR^d @>A^T
>>  \RR^k @>\beta>> W @>>> 0.
\end{CD}
\end{equation*}
Choosing a basis for $\ker(A)$, we can identify $W \simeq \RR^{m-d}$.
The columns of the matrix $\beta$ define a vector configuration $B
=\{b_1,\dots,b_m\} \subset \RR^{m-d}$ called a {\em Gale dual
  configuration} of~$A$.  Projecting into $W$, we see that $\mathcal
C_{i,\ww}$ equals the set of all maximal chains $J$ such that
$\beta(\ww)$ lies on the cone $\RR_{> 0} \bigl\{ \beta({e_{J_1}}),
\ldots, \beta(e_{J_{m-d-1}}), -\beta(e_i)\bigr\}$.
It follows that 
\[
\mathcal{J} \in \mathcal C_{i,\ww} \quad \text{if and only if}\quad
\beta(\ww) = \sum_{j=1}^m w_j b_j \,\,\in \,\, \RR_{> 0} \bigl\{
\sigma_{J_1}, \ldots, \sigma_{J_{m-d-1}}, -b_i\bigr\},
\]
where $\sigma_{\mathcal{J}} \,:= \,\beta(e_{\mathcal{J}}) \,\,=
\,\,\sum_{j\in \mathcal{J}} b_j$.

We can also restate the definition of the tropical matroid strata in
terms of Gale duals.  Namely, there exists a non-zero constant $c$,
depending only on $d$, $m$ and our choice of Gale dual $B$, such that,
given a maximal chain of flats
$\mathcal{J}$ 
in the matroid $M^*(A)$, we have:
\begin{equation}
{\det} (\, A^T, e_{J_1}, e_{J_2}, \ldots, e_{J_{m-d-1}}, e_i \,)\,\,
= \,\,c\cdot{\det} (\, \sigma_{J_1}, \sigma_{J_2}, \ldots, \sigma_{J_{m-d-1}}, b_i \,).\label{eq:3}
\end{equation}
Hence, the tropical matroid strata in $G(d,\{[m]\}) $ are determined
by the signs of the determinant on the right-hand side
of~\eqref{eq:3}.  If $\mathcal{J}\in \mathcal C_{i,\ww}$ for generic
$\ww \in \RR^m$, then the vectors $\{\sigma_{J_1}, \ldots,
\sigma_{J_{m-d-1}}, b_i\}$ in $\RR^{m-d}$ are linearly independent.
Let $M(\mathcal{J},B,i)$ be the matrix whose columns are these
vectors.  Then, $\mathcal{J}\in \mathcal C_{i,\ww}$ if and only if the
vector $x = M(\mathcal{J},B,i)^{-1} \beta(\ww)$ has positive entries.
By Cramer's rule, those entries~are
\begin{equation}\label{eq:cramer}
\left\{    \begin{aligned}
x_k & \,\,= \,\,\frac
 {\det (\sigma_{J_1}, \ldots, \sigma_{J_{k-1}}, \beta(\ww), \sigma_{J_{k+1}},
  \ldots, \sigma_{J_{m-d-1}}, -b_i) }{\det(M(\mathcal{J},B,i))} \qquad
\operatorname{ for} \,\, 0\leq k<m-d,\\
 x_m & \,\,= \,\,\frac
 {\det  (\sigma_{J_1}, \sigma_{J_2},\ldots,  \sigma_{J_{m-d-1}}, \beta(\ww)) }{\det(M(\mathcal{J},B,i))}.
  \end{aligned}\right.
 \end{equation} 

 Suppose now that $A$ and $A'$ are two configurations in the same
 tropical matroid stratum.  Let $B$ and $B'$ be their Gale duals.
 Then $M^*(A) = M^*(A')$ and the denominators
 $\det(M(\mathcal{J},B,i))$ and $\det(M(\mathcal{J},B',i))$ 
 in~\eqref{eq:cramer} have the same signs.  On the other hand, let us
 consider the oriented hyperplane arrangement in $\RR^{m-d}$
 consisting of the hyperplanes
 $\Hcal_{\mathcal{J},B,k,i}=\langle\sigma_{J_1}, \ldots,
 \sigma_{J_{k-1}}, \sigma_{J_{k+1}}, \ldots, \sigma_{J_{m-d-1}},
 b_i\rangle$, for $1\leq k \leq m-d-1, \ i \not \in J_{m-d-1}$, as well
 as the hyperplane $\Hcal_{\mathcal{J}}=\langle\sigma_{J_1}, \ldots,
 \sigma_{J_{m-d-1}}\rangle$, for all maximal chains $\mathcal{J} \in
 M^*(A)$ such that $ \sigma_{J_1}, \ldots, \sigma_{J_{m-d-1}}$ are
 linearly independent.  The signs of the numerators in~\eqref{eq:cramer} are determined by the oriented hyperplane
 arrangement just defined.  Since $M^*(A) = M^*(A')$, we can establish
 a correspondence between the cells of the complements of these
 arrangements that preserves the signs in~\eqref{eq:cramer} for both
 $A$ and $A'$, given weights $\ww$ and $\ww'$ in corresponding cells.
 This means that $\mathcal C_{i,\ww} = \mathcal C_{i,\ww'}$ as we
 wanted to show.
 \end{proof}

We note that the conclusion of Theorem~\ref{thm:degreeDelta} is also valid 
on tropical matroid strata  where $A$ is defective. In that case the $A$-discriminant
$\Delta_A$ equals $1$, and its degree is the zero vector. 
 We end this section by showing how to obtain 
 our main result on mixed discriminants.
 
 \begin{proof}[Proof of  Theorem~\ref{thm:degreeMixed}]
 Suppose  that $A$ is the Cayley matrix of $n$ configurations $A_1,\ldots,A_n$
and let  ${\mathcal I} =  \{I_1, \ldots, I_n\}$ be the associated partition of $[m]$.
It follows from~\eqref{eq:degree1} that
$$\deg_{A_k}(\Delta_A) \ =\  \sum_{i\in I_k}\sum_{J \in {\mathcal C}_{i,\ww} } |\, 
\det ( A^T, e_{J_1}, \ldots, e_{J_{m-d-1}}, e_i ) \,| \, .$$ By the
same argument as in the proof of Theorem~\ref{thm:degreeDelta}, we
conclude that the above expression defines a fixed linear form on
$\wedge^d \RR^m$ for all matrices $A$ in a fixed tropical matroid
stratum.
\end{proof}

In closing, we wish to reiterate that
combining Theorem~\ref{thm:degreeMixed} with
Rinc\'on's results in \cite{Rincon} leads to 
powerful algorithms for computing piecewise polynomial
degree formulas. Here is an example that illustrates this.
We consider the $n$-dimensional version of the system
\eqref{eq:deltadelta22}:
\begin{equation*}
f_i \, = c_{i0} + c_{i1} x_1^{d_i} + c_{i2} x_2^{d_i} + \cdots + c_{in} x_n^{d_i} 
\quad \hbox{for} \,\, \, i = 1,2,\ldots,n, 
\end{equation*}
where $ 0 \leq d_1 \leq d_2 \leq \cdots \leq d_n$ are
coprime integers.
The Cayley matrix $A$ has $2n$ rows and $n^2+n$ columns.
Using his software, Felipe
Rinc\'on computed the corresponding tropical discriminant
for $n=4$, while keeping the $d_i$ as unknowns, and he~found 
\begin{equation*}
{\rm deg}_{A_i} (\Delta_{A_1,\ldots,A_n}) \,\,\, = \,\,\,
d_1 \cdots d_{i-1} d_{i+1} \cdots d_n \cdot \bigl(\,d_i + (-n)d_1 + d_2 + d_3 + \cdots + d_n \, \bigr).
\end{equation*}
Thus, we have a computational proof of this formula for $n \leq 4$, 
and it remains a conjecture for $n \geq 5$.
This shows how the findings of this section 
may be used in experimental mathematics.

\medskip
\hfill

\noindent {\bf Acknowledgments:} {MAC was supported by an AXA Mittag-Leffler
  postdoctoral fellowship (Sweden) and an NSF postdoctoral
  fellowship DMS-1103857 (USA).  AD was supported by UBACYT
  20020100100242, CONICET PIP 112-200801-00483 and ANPCyT 2008-0902
  (Argentina). SDR was partially supported by VR grant  NT:2010-5563
  (Sweden). BS was supported by NSF grants
  DMS-0757207 and DMS-0968882 (USA). 
  This project started at the Institut Mittag-Leffler 
during  the Spring 2011 program on
  ``Algebraic Geometry with a View Towards Applications.''
  We thank IML for its wonderful hospitality.}

\bigskip
\bigskip



\bigskip
\bigskip

\begin{small}
\noindent Authors' e-mail:
{\tt cattani@math.umass.edu}, 
{\tt macueto@math.columbia.edu},  \\ 
{\tt alidick@dm.uba.ar}, 
{\tt dirocco@math.kth.se},
{\tt bernd@math.berkeley.edu}
\end{small}

\end{document}